\documentclass[a4paper,twoside,fleqn,notitlepage]{article}%
\usepackage{amsmath}
\usepackage{amsfonts}
\usepackage{amssymb}
\usepackage{graphicx}%
\providecommand{\U}[1]{\protect\rule{.1in}{.1in}}
\newtheorem{theorem}{Theorem}

\newtheorem{corollary}[theorem]{Corollary}

\newtheorem{definition}[theorem]{Definition}

\newtheorem{lemma}[theorem]{Lemma}

\newtheorem{proposition}[theorem]{Proposition}
\newtheorem{remark}[theorem]{Remark}

\newenvironment{proof}[1][Proof]{\noindent\textbf{#1.} }{\ \rule{0.5em}{0.5em}}
\begin{document}

\title{On a class of $q$-Bernoulli, $q$-Euler and $q$-Genocchi polynomials}
\author{N. I. Mahmudov, M. Momenzadeh\\Eastern Mediterranean University\\Gazimagusa, TRNC, Mersiin 10, Turkey \\Email: nazim.mahmudov@emu.edu.tr\\$\ \ \ \ \ \ \ \ \ $mohammed.momenzadeh@emu.edu.tr}
\maketitle

\begin{abstract}
The main purpose of this paper is to introduce and investigate a class of
$q$-Bernoulli, $q$-Euler and $q$-Genocchi polynomials. The $q$-analogues of
well-known formulas are derived. The $q$-analogue of the
Srivastava--Pint\'{e}r addition theorem is obtained. Some new identities
involving $q$-polynomials are proved.

\end{abstract}

\section{ Introduction}

Throughout this paper, we always make use of the classical definition of
quantum concepts as follows:

The $q$-shifted factorial is defined by%
\begin{align*}
\left(  a;q\right)  _{0} &  =1,\ \ \ \left(  a;q\right)  _{n}=%
{\displaystyle\prod\limits_{j=0}^{n-1}}
\left(  1-q^{j}a\right)  ,\ \ \ n\in\mathbb{N},\\
\left(  a;q\right)  _{\infty} &  =%
{\displaystyle\prod\limits_{j=0}^{\infty}}
\left(  1-q^{j}a\right)  ,\ \ \ \ \left\vert q\right\vert <1,\ \ a\in
\mathbb{C}.
\end{align*}
It is known that%
\[
\left(  a;q\right)  _{n}=\sum_{k=0}^{n}\left[
\begin{array}
[c]{c}%
n\\
k
\end{array}
\right]  _{q}q^{\frac{1}{2}k\left(  k-1\right)  }\left(  -1\right)  ^{k}a^{k}.
\]
The $q$-numbers and $q$-numbers factorial and their improved forms are defined
by%
\begin{align*}
\left[  a\right]  _{q} &  =\frac{1-q^{a}}{1-q},\ \ \ \left(  q\neq
1,\ a\in\mathbb{C}\right)  ;\ \ \ \\
\ \ \left[  0\right]  _{q}! &  =1,\left[  n\right]  _{q}!=\left[  n\right]
_{q}\left[  n-1\right]  _{q}!\ \ \ \ \ ,\ .
\end{align*}
The $q$-polynomail coefficient and improved type of them are defined by%
\[
\left[
\begin{array}
[c]{c}%
n\\
k
\end{array}
\right]  _{q}=\frac{\left(  q;q\right)  _{n}}{\left(  q;q\right)
_{n-k}\left(  q;q\right)  _{k}},\ \ \ \ \ (k\leqslant n,k,n\in%
\mathbb{N}
)
\]
In the standard approach to the $q$-calculus two exponential function are
used, these $q$-exponential and improved type (see \cite{cel}) of it are
defined as follows:%

\begin{align*}
e_{q}\left(  z\right)   &  =\sum_{n=0}^{\infty}\frac{z^{n}}{\left[  n\right]
_{q}!}=\prod_{k=0}^{\infty}\frac{1}{\left(  1-\left(  1-q\right)
q^{k}z\right)  },\ \ \ 0<\left\vert q\right\vert <1,\ \left\vert z\right\vert
<\frac{1}{\left\vert 1-q\right\vert },\ \ \ \ \ \ \ \\
E_{q}(z)  &  =e_{1/q}\left(  z\right)  =\sum_{n=0}^{\infty}\frac{q^{\frac
{1}{2}n\left(  n-1\right)  }z^{n}}{\left[  n\right]  _{q}!}=\prod
_{k=0}^{\infty}\left(  1+\left(  1-q\right)  q^{k}z\right)
,\ \ \ \ \ \ \ 0<\left\vert q\right\vert <1,\ z\in\mathbb{C},\\
\mathcal{E}_{q}\left(  z\right)   &  =e_{q}\left(  \frac{z}{2}\right)
E_{q}\left(  \frac{z}{2}\right)  =%
{\displaystyle\sum\limits_{n=0}^{\infty}}
\frac{(-1,q)_{n}}{2^{n}}\frac{z^{n}}{\left[  n\right]  _{q}!}=%
{\displaystyle\sum\limits_{n=0}^{\infty}}
\frac{z^{n}}{\left\{  n\right\}  _{q}!}\\
&  =\prod_{k=0}^{\infty}\frac{\left(  1+\left(  1-q\right)  q^{k}\frac{z}%
{2}\right)  }{\left(  1-\left(  1-q\right)  q^{k}\frac{z}{2}\right)
},0<q<1,\left\vert z\right\vert <\frac{2}{1-q}.
\end{align*}
The form of improved type of $q$-exponential function $\mathcal{E}_{q}\left(
z\right)  $, motivate us to define a new $q$-addition and $q$-substraction as
follows:%
\begin{align*}
\left(  x\oplus_{q}y\right)  ^{n}  &  :=%
{\displaystyle\sum\limits_{k=0}^{n}}
\left[
\begin{array}
[c]{c}%
n\\
k
\end{array}
\right]  _{q}\frac{(-1,q)_{k}(-1,q)_{n-k}}{2^{n}}x^{k}y^{n-k}%
,\ \ \ n=0,1,2,...,\\
\left(  x\ominus_{q}y\right)  ^{n}  &  :=%
{\displaystyle\sum\limits_{k=0}^{n}}
\left[
\begin{array}
[c]{c}%
n\\
k
\end{array}
\right]  _{q}\frac{(-1,q)_{k}(-1,q)_{n-k}}{2^{n}}x^{k}\left(  -y\right)
^{n-k},\ \ \ n=0,1,2,...
\end{align*}
It follows that%
\[
\mathcal{E}_{q}\left(  tx\right)  \mathcal{E}_{q}\left(  ty\right)
=\sum_{n=0}^{\infty}\left(  x\oplus_{q}y\right)  ^{n}\frac{t^{n}}{\left[
n\right]  _{q}!}.
\]
The Bernoulli numbers $\left\{  B_{m}\right\}  _{m\geq0}$ are rational numbers
in a sequence defined by the binomial recursion formula
\begin{equation}
\sum_{k=0}^{m}\left(
\begin{array}
[c]{c}%
m\\
k
\end{array}
\right)  B_{k}-B_{m}=\left\{
\begin{tabular}
[c]{ll}%
$1,$ & $m=1,$\\
$0,$ & $m>1,$%
\end{tabular}
\ \ \ \ \ \ \ \ \ \right.  \label{b1}%
\end{equation}
or equivalently, the generating function%
\[
\sum_{k=0}^{\infty}B_{k}\frac{t^{k}}{k!}=\frac{t}{e^{t}-1}.
\]
$q$-analogues of the Bernoulli numbers were first studied by Carlitz
\cite{calitz1} in the middle of the last century when he introduced a new
sequence $\left\{  \beta_{m}\right\}  _{m\geqslant0}$:%
\begin{equation}
\sum_{k=0}^{m}\left(
\begin{array}
[c]{c}%
m\\
k
\end{array}
\right)  \beta_{k}q^{k+1}-\beta_{m}=\left\{
\begin{tabular}
[c]{ll}%
$1,$ & $m=1,$\\
$0,$ & $m>1.$%
\end{tabular}
\ \ \ \ \ \ \ \ \ \right.  \label{b2}%
\end{equation}
Here, and in the remainder of the paper, the parameter we make the assumption
that $\left\vert q\right\vert <1.$Clearly we recover (\ref{b1}) if we let
$q\rightarrow1$ in (\ref{b2}).The $q$-binomial formula is known as%
\begin{align*}
\left(  1\ominus_{q}x\right)  ^{n}  &  =\sum_{k=0}^{n}\left[
\begin{array}
[c]{c}%
n\\
k
\end{array}
\right]  _{q}\frac{(-1,q)_{k}}{2^{k}}(-x)^{k}=\sum_{k=0}^{n}\left[
\begin{array}
[c]{c}%
n\\
k
\end{array}
\right]  _{q}\frac{(1+1)(1+q)...(1+q^{k-1})x^{k}}{2^{k}}(-1)^{k}\\
\left(  1-a\right)  _{q}^{n}  &  =\left(  a;q\right)  _{n}=%
{\displaystyle\prod\limits_{j=0}^{n-1}}
\left(  1-q^{j}a\right)  =\sum_{k=0}^{n}\left[
\begin{array}
[c]{c}%
n\\
k
\end{array}
\right]  _{q}q^{\frac{1}{2}k\left(  k-1\right)  }\left(  -1\right)  ^{k}a^{k}.
\end{align*}
The above $q$-standard notation can be found in \cite{andrew}.

Carlitz has introduced the $q$-Bernoulli numbers and polynomials in
\cite{carlitz}. Srivastava and Pint\'{e}r proved some relations and theorems
between the Bernoulli polynomials and Euler polynomials in \cite{sri1}. They
also gave some generalizations of these polynomials. In \cite{kim2}%
-\cite{kim7}, Kim et al. investigated some properties of the $q$-Euler
polynomials and Genocchi polynomials. They gave some recurrence relations. In
\cite{cenkci}, Cenkci et al. gave the $q$-extension of Genocchi numbers in a
different manner. In \cite{kim5}, Kim gave a new concept for the $q$-Genocchi
numbers and polynomials. In \cite{simsek}, Simsek et al. investigated the
$q$-Genocchi zeta function and $l$-function by using generating functions and
Mellin transformation. There are numerous recent studies on this subject by
among many other authors: Cenkci et al. \cite{cenkci}, \cite{cenkci2}, Choi et
al \cite{choi}, Cheon \cite{cheon}, Luo and Srivastava \cite{luo},
\cite{luo2}, \cite{luo3}, Srivastava et al.\cite{sri1}, \cite{sri2}, Nalci and
Pashaev \cite{pash} Gabouary and Kurt B., \cite{kurt1}, Kim et al.
\cite{kimd}, Kurt V. \cite{kurt}.

We first give here the definitions of the $q$-numbers and $q$-polynomials as
follows. It should be mentioned that the definition of $q$-Bernoulli numbers
in Definition \ref{D:1} can br found in \cite{pash}.

\begin{definition}
\label{D:1}Let $q\in\mathbb{C},\ 0<\left\vert q\right\vert <1.$ The
$q$-Bernoulli numbers $\mathfrak{b}_{n,q}$ and polynomials $\mathfrak{B}%
_{n,q}\left(  x,y\right)  $ are defined by the means of the generating
functions:%
\begin{align*}
\widehat{\mathfrak{B}}\left(  t\right)   &  :=\frac{te_{q}\left(  -\frac{t}%
{2}\right)  }{e_{q}\left(  \frac{t}{2}\right)  -e_{q}\left(  -\frac{t}%
{2}\right)  }=\frac{t}{\mathcal{E}_{q}\left(  t\right)  -1}=\sum_{n=0}%
^{\infty}\mathfrak{b}_{n,q}\frac{t^{n}}{\left[  n\right]  _{q}!}%
,\ \ \ \left\vert t\right\vert <2\pi,\\
\frac{t}{\mathcal{E}_{q}\left(  t\right)  -1}\mathcal{E}_{q}\left(  tx\right)
\mathcal{E}_{q}\left(  ty\right)   &  =\sum_{n=0}^{\infty}\mathfrak{B}%
_{n,q}\left(  x,y\right)  \frac{t^{n}}{\left[  n\right]  _{q}!}%
,\ \ \ \left\vert t\right\vert <2\pi.
\end{align*}

\end{definition}

\begin{definition}
\label{D:2}Let $q\in\mathbb{C},\ 0<\left\vert q\right\vert <1.$ The $q$-Euler
numbers $\mathfrak{e}_{n,q}$ and polynomials $\mathfrak{E}_{n,q}\left(
x,y\right)  $ are defined by the means of the generating functions:%
\begin{align*}
\widehat{\mathfrak{E}}\left(  t\right)   &  :=\frac{2e_{q}\left(  -\frac{t}%
{2}\right)  }{e_{q}\left(  \frac{t}{2}\right)  +e_{q}\left(  -\frac{t}%
{2}\right)  }=\frac{2}{\mathcal{E}_{q}\left(  t\right)  +1}=\sum_{n=0}%
^{\infty}\mathfrak{e}_{n,q}\frac{t^{n}}{\left[  n\right]  _{q}!}%
,\ \ \ \left\vert t\right\vert <\pi,\\
\frac{2}{\mathcal{E}_{q}\left(  t\right)  +1}\mathcal{E}_{q}\left(  tx\right)
\mathcal{E}_{q}\left(  ty\right)   &  =\sum_{n=0}^{\infty}\mathfrak{E}%
_{n,q}\left(  x,y\right)  \frac{t^{n}}{\left[  n\right]  _{q}!}%
,\ \ \ \left\vert t\right\vert <\pi.
\end{align*}

\end{definition}

\begin{definition}
\label{D:3}Let $q\in\mathbb{C},\ 0<\left\vert q\right\vert <1.$ The
$q$-Genocchi numbers $\mathfrak{g}_{n,q}$ and polynomials $\mathfrak{G}%
_{n,q}\left(  x,y\right)  $ are defined by the means of the generating
functions:%
\begin{align*}
\widehat{\mathfrak{G}}\left(  t\right)   &  :=\frac{2te_{q}\left(  -\frac
{t}{2}\right)  }{e_{q}\left(  \frac{t}{2}\right)  +e_{q}\left(  -\frac{t}%
{2}\right)  }=\frac{2t}{\mathcal{E}_{q}\left(  t\right)  +1}=\sum
_{n=0}^{\infty}\mathfrak{g}_{n,q}\frac{t^{n}}{\left[  n\right]  _{q}%
!},\ \ \ \left\vert t\right\vert <\pi,\\
\frac{2t}{\mathcal{E}_{q}\left(  t\right)  +1}\mathcal{E}_{q}\left(
tx\right)  \mathcal{E}_{q}\left(  ty\right)   &  =\sum_{n=0}^{\infty
}\mathfrak{G}_{n,q}\left(  x,y\right)  \frac{t^{n}}{\left[  n\right]  _{q}%
!},\ \ \ \left\vert t\right\vert <\pi.
\end{align*}

\end{definition}

\begin{definition}
\label{D:4}Let $q\in\mathbb{C},\ 0<\left\vert q\right\vert <1.$ The
$q$-tangent numbers $\mathfrak{T}_{n,q}$ are defined by the means of the
generating functions:%
\begin{align*}
\tanh_{q}t  &  =-i\tan_{q}\left(  it\right)  =\frac{e_{q}\left(  t\right)
-e_{q}\left(  -t\right)  }{e_{q}\left(  t\right)  +e_{q}\left(  -t\right)
}=\frac{\mathcal{E}_{q}\left(  2t\right)  -1}{\mathcal{E}_{q}\left(
2t\right)  +1}\\
&  =\sum_{n=1}^{\infty}\mathfrak{T}_{2n+1,q}\frac{\left(  -1\right)
^{k}t^{2n+1}}{\left[  2n+1\right]  _{q}!}.
\end{align*}

\end{definition}

It is obvious that by tending $q$ to 1 from the left side, we lead to the
classic definition of these polynomials:%
\begin{align*}
\mathfrak{b}_{n,q}  &  :=\mathfrak{B}_{n,q}\left(  0\right)  ,\ \ \ \lim
_{q\rightarrow1^{-}}\mathfrak{B}_{n,q}\left(  x\right)  =B_{n}\left(
x\right)  ,\ \ \ \lim_{q\rightarrow1^{-}}\mathfrak{B}_{n,q}\left(  x,y\right)
=B_{n}\left(  x+y\right)  ,\ \ \ \lim_{q\rightarrow1^{-}}\mathfrak{b}%
_{n,q}=B_{n},\\
\mathfrak{e}_{n,q}  &  :=\mathfrak{E}_{n,q}\left(  0\right)  ,\ \ \ \lim
_{q\rightarrow1^{-}}\mathfrak{E}_{n,q}\left(  x\right)  =E_{n}\left(
x\right)  ,\ \ \ \lim_{q\rightarrow1^{-}}\mathfrak{E}_{n,q}\left(  x,y\right)
=E_{n}\left(  x+y\right)  ,\ \ \ \lim_{q\rightarrow1^{-}}\mathfrak{e}%
_{n,q}=E_{n},\\
\mathfrak{g}_{n,q}  &  :=\mathfrak{G}_{n,q}\left(  0\right)  ,\ \ \ \lim
_{q\rightarrow1^{-}}\mathfrak{G}_{n,q}\left(  x\right)  =G_{n}\left(
x\right)  ,\ \ \ \lim_{q\rightarrow1^{-}}\mathfrak{G}_{n,q}\left(  x,y\right)
=G_{n}\left(  x+y\right)  \ \ \ \lim_{q\rightarrow1^{-}}\mathfrak{g}%
_{n,q}=G_{n}.
\end{align*}
Here $B_{n}\left(  x\right)  ,$ $E_{n}\left(  x\right)  $ and $G_{n}\left(
x\right)  $ denote the classical Bernoulli, Euler and Genocchi polynomials
which are defined by%
\[
\frac{t}{e^{t}-1}e^{tx}=\sum_{n=0}^{\infty}B_{n}\left(  x\right)  \frac{t^{n}%
}{n!},\ \ \ \text{\ }\frac{2}{e^{t}+1}e^{tx}=\sum_{n=0}^{\infty}E_{n}\left(
x\right)  \frac{t^{n}}{n!}\ \ \text{and\ \ \ \ }\frac{2t}{e^{t}+1}e^{tx}%
=\sum_{n=0}^{\infty}G_{n}\left(  x\right)  \frac{t^{n}}{n!}.
\]

The aim of the present paper is to obtain some results for the above newly
defined $q$-polynomials. It should be mentioned that $q$-Bernoulli and
$q$-Euler polynomials in our definitions are polynomials of $x$ and $y$ and
when $y=0$ they are polynomials of $x$, but in other definitions they respect
to $q^{x}$. First advantage of this approach is that for $q\rightarrow1^{-}$
,$\mathfrak{B}_{n,q}\left(  x,y\right)  $ ($\mathfrak{E}_{n,q}\left(
x,y\right)  ,$ $\mathfrak{G}_{n,q}\left(  x,y\right)  $) becomes the classical
Bernoulli $\mathfrak{B}_{n}\left(  x+y\right)  $ (Euler $\mathfrak{E}%
_{n}\left(  x+y\right)  ,\ $Genocchi $\mathfrak{G}_{n,q}\left(  x,y\right)  $)
polynomial and we may obtain the $q$-analogues of well-known results, for
example Srivastava and Pint\'{e}r \cite{pinter}, Cheon \cite{cheon}, etc.
Second advantage is that similar to the classical case odd numbered terms of
the Bernoulli numbers $\mathfrak{b}_{k,q}$ and the Genocchi numbres
$\mathfrak{g}_{k,q}$are zero, and even numbered terms of the Euler numbers
$\mathfrak{e}_{n,q}$ are zero.

\section{Preliminary results}

In this section we shall provide some basic formulas for the $q$-Bernoulli,
$q$-Euler and $q$-Genocchi numbers and polynomials in order to obtain the main
results of this paper in the next section.

\begin{lemma}
\label{L:11}The $q$-Bernoulli numbers $\mathfrak{b}_{n,q}$ satisfy the
following $q$-binomial recurrence:
\begin{equation}%
{\displaystyle\sum\limits_{k=0}^{n}}
\left[
\begin{array}
[c]{c}%
n\\
k
\end{array}
\right]  _{q}\frac{(-1,q)_{n-k}}{2^{n-k}}\mathfrak{b}_{k,q}-\mathfrak{b}%
_{n,q}=\left\{
\begin{tabular}
[c]{ll}%
$1,$ & $n=1,$\\
$0,$ & $n>1.$%
\end{tabular}
\ \ \ \ \ \ \ \ \ \ \ \right.  \label{ber}%
\end{equation}

\end{lemma}

\begin{proof}
By simple multiplication on (\ref{D:1}) we see that
\[
\widehat{\mathfrak{B}}\left(  t\right)  \mathcal{E}_{q}\left(  t\right)
=t+\widehat{\mathfrak{B}}\left(  t\right)  .
\]
So%
\[%
{\displaystyle\sum\limits_{n=0}^{\infty}}
{\displaystyle\sum\limits_{k=0}^{n}}
\left[
\begin{array}
[c]{c}%
n\\
k
\end{array}
\right]  _{q}\frac{(-1,q)_{n-k}}{2^{n-k}}\mathfrak{b}_{k,q}\frac{t^{n}%
}{\left[  n\right]  _{q}!}=t+%
{\displaystyle\sum\limits_{n=0}^{\infty}}
\mathfrak{b}_{n,q}\ \frac{t^{n}}{\left[  n\right]  _{q}!}.
\]
The statement follows by comparing $t^{m}$-coefficients.
\end{proof}

We use this formula to calculate the first few $\mathfrak{b}_{k,q}$.%
\[%
\begin{tabular}
[c]{ll}%
$\mathfrak{b}_{0,q}=$ & $1,$\\
$\mathfrak{b}_{1,q}=$ & $-\frac{1}{2}=-\frac{1}{\left\{  2\right\}  _{q}},$\\
$\mathfrak{b}_{2,q}=$ & $\frac{1}{4}\dfrac{q(q+1)}{q^{2}+q+1}=\frac{q[2]_{q}%
}{4[3]_{q}},$\\
$\mathfrak{b}_{3,q}=$ & $0.$%
\end{tabular}
\ \ \ \ \
\]

The similar property can be proved for $q$-Euler numbers%
\begin{equation}%
{\displaystyle\sum\limits_{k=0}^{m}}
\left\{
\begin{array}
[c]{c}%
m\\
k
\end{array}
\right\}  _{q}\mathfrak{e}_{k,q}+\mathfrak{e}_{m,q}=\left\{
\begin{tabular}
[c]{ll}%
$2,$ & $m=0,$\\
$0,$ & $m>0.$%
\end{tabular}
\ \ \ \ \ \ \ \ \ \ \right.  \label{euler}%
\end{equation}
\bigskip and $q$-Genocchi numbers%
\begin{equation}%
{\displaystyle\sum\limits_{k=0}^{m}}
\left\{
\begin{array}
[c]{c}%
m\\
k
\end{array}
\right\}  _{q}\mathfrak{g}_{k,q}+\mathfrak{g}_{m,q}=\left\{
\begin{tabular}
[c]{ll}%
$2,$ & $m=1,$\\
$0,$ & $m>1.$%
\end{tabular}
\ \ \ \ \ \ \ \ \ \ \right.  \label{gen}%
\end{equation}

Using the above recurrence formulae we calculate the first few $\mathfrak{e}%
_{n,q}$ and $\mathfrak{g}_{n,q}$as well.%

\[%
\begin{tabular}
[c]{ll}%
$\mathfrak{e}_{0,q}=$ & $1,$\\
$\mathfrak{e}_{1,q}=$ & $-\dfrac{1}{2},$\\
$\mathfrak{e}_{2,q}=$ & $0,$\\
$\mathfrak{e}_{3,q}=$ & $\frac{[3]_{q}[2]_{q}-[4]_{q}}{8}=\dfrac{q\left(
1+q\right)  }{8},$%
\end{tabular}%
\begin{tabular}
[c]{ll}%
$\mathfrak{g}_{0,q}=$ & $0,$\\
$\mathfrak{g}_{1,q}=$ & $1,$\\
$\mathfrak{g}_{2,q}=$ & $-\frac{[2]_{q}}{2}=-\dfrac{q+1}{2},$\\
$\mathfrak{g}_{3,q}=$ & $0.$%
\end{tabular}
\ \ \ \ \ \ \
\]

\begin{remark}
The first advantage of the new $q$-numbers $\mathfrak{b}_{k,q},$
$\mathfrak{e}_{k,q}$ and $\mathfrak{g}_{k,q}$ is that similar to classical
case odd numbered terms of the Bernoulli numbers $\mathfrak{b}_{k,q}$ and the
Genocchi numbres $\mathfrak{g}_{k,q}$are zero, and even numbered terms of the
Euler numbers $\mathfrak{e}_{n,q}$ are zero.
\end{remark}

Next lemma gives the relationsheep between $q$-Genocchi numbers and
$q$-Tangent numbers.

\begin{lemma}
Fro any $n\in\mathbb{N}$ we have%
\[
\mathfrak{T}_{2n+1,q}=\mathfrak{g}_{2n+2,q}\frac{\left(  -1\right)
^{k-1}2^{2n+1}}{\left[  2n+2\right]  _{q}}.
\]

\end{lemma}

\begin{proof}
First we recall the definition of $q$-trigonometric functions.%
\begin{align*}
\cos_{q}t  &  =\frac{e_{q}\left(  it\right)  +e_{q}\left(  -it\right)  }%
{2},\ \ \ \ \ \sin_{q}t=\frac{e_{q}\left(  it\right)  -e_{q}\left(
-it\right)  }{2i},\\
i\tan_{q}t  &  =\frac{e_{q}\left(  it\right)  -e_{q}\left(  -it\right)
}{e_{q}\left(  it\right)  +e_{q}\left(  -it\right)  },\ \ \ \ \ \ \ \cot
_{q}t=i\frac{e_{q}\left(  it\right)  +e_{q}\left(  -it\right)  }{e_{q}\left(
it\right)  -e_{q}\left(  -it\right)  }.
\end{align*}
Now by choosing $z=2it$ in $\widehat{\mathfrak{B}}\left(  z\right)  $, we get%
\[
\widehat{\mathfrak{B}}\left(  2it\right)  =\frac{2it}{\mathcal{E}_{q}\left(
2it\right)  -1}=\frac{te_{q}\left(  -it\right)  }{\sin_{q}t}=\sum
_{n=0}^{\infty}\mathfrak{b}_{n,q}\frac{\left(  2it\right)  ^{n}}{\left[
n\right]  _{q}!}.
\]
It follows that%
\begin{align*}
\widehat{\mathfrak{B}}\left(  2it\right)   &  =\frac{te_{q}\left(  -it\right)
}{\sin_{q}t}=\frac{t}{\sin_{q}t}\left(  \cos_{q}t-i\sin_{q}t\right)
=t\cot_{q}t-it\\
&  =\mathfrak{b}_{0,q}+2it\mathfrak{b}_{1,q}+\sum_{n=2}^{\infty}%
\mathfrak{b}_{n,q}\frac{\left(  2it\right)  ^{n}}{\left[  n\right]  _{q}!}\\
&  =1-it+\sum_{n=2}^{\infty}\mathfrak{b}_{n,q}\frac{\left(  2it\right)  ^{n}%
}{\left[  n\right]  _{q}!}.
\end{align*}
Since $t\cot_{q}t$ is even in the above sum odd coefficients $\mathfrak{b}%
_{2k+1,q}$ , $k=1,2,...$are zero we get
\[
t\cot_{q}t=1+\sum_{n=2}^{\infty}\mathfrak{b}_{n,q}\frac{\left(  2it\right)
^{n}}{\left[  n\right]  _{q}!}=1+\sum_{n=1}^{\infty}\mathfrak{b}_{n,q}%
\frac{\left(  2it\right)  ^{2n}}{\left[  2n\right]  _{q}!}.
\]
\bigskip By choosing $z=2it$ in $\widehat{\mathfrak{G}}\left(  z\right)  $, we
get%
\[
\widehat{\mathfrak{G}}\left(  2it\right)  =\frac{4it}{\mathcal{E}_{q}\left(
2it\right)  +1}=\frac{2ite_{q}\left(  -it\right)  }{\cos_{q}t}=\sum
_{n=0}^{\infty}\mathfrak{g}_{n,q}\frac{\left(  2it\right)  ^{n}}{\left[
n\right]  _{q}!}.
\]%
\begin{align*}
\widehat{\mathfrak{G}}\left(  2it\right)   &  =\frac{4it}{\mathcal{E}%
_{q}\left(  2it\right)  +1}=\frac{2ite_{q}\left(  -it\right)  }{\cos_{q}%
t}=\frac{2it}{\cos_{q}t}\left(  \cos_{q}t-i\sin_{q}t\right)  =2it+2t\tan
_{q}t\\
&  =\mathfrak{g}_{0,q}+2it\mathfrak{g}_{1,q}+\sum_{n=2}^{\infty}%
\mathfrak{g}_{n,q}\frac{\left(  2it\right)  ^{n}}{\left[  n\right]  _{q}!}\\
&  =2it+\sum_{n=2}^{\infty}\mathfrak{g}_{n,q}\frac{\left(  2it\right)  ^{n}%
}{\left[  n\right]  _{q}!}.
\end{align*}
It follows that%
\begin{align*}
2t\tan_{q}t  &  =\sum_{n=1}^{\infty}\mathfrak{g}_{2n,q}\frac{\left(
2it\right)  ^{2n}}{\left[  2n\right]  _{q}!},\ \ \ \ \ \ \tan_{q}t=\sum
_{n=1}^{\infty}\mathfrak{g}_{2n,q}\frac{\left(  -1\right)  ^{n}\left(
2t\right)  ^{2n-1}}{\left[  2n\right]  _{q}!}\\
\tanh_{q}t  &  =-i\tan_{q}\left(  it\right)  =-i\sum_{n=1}^{\infty
}\mathfrak{g}_{2n,q}\frac{\left(  -1\right)  ^{n}\left(  2it\right)  ^{2n-1}%
}{\left[  2n\right]  _{q}!}\\
&  =-\sum_{n=1}^{\infty}\mathfrak{g}_{2n,q}\frac{\left(  2t\right)  ^{2n-1}%
}{\left[  2n\right]  _{q}!}=-\sum_{n=1}^{\infty}\mathfrak{g}_{2n+2,q}%
\frac{\left(  2t\right)  ^{2n+1}}{\left[  2n+2\right]  _{q}!}.
\end{align*}
Thus%
\begin{align*}
\tanh_{q}t  &  =-i\tan_{q}\left(  it\right)  =\frac{e_{q}\left(  t\right)
-e_{q}\left(  -t\right)  }{e_{q}\left(  t\right)  +e_{q}\left(  -t\right)
}=\frac{\mathcal{E}_{q}\left(  2t\right)  -1}{\mathcal{E}_{q}\left(
2t\right)  +1}\\
&  =\sum_{n=1}^{\infty}\mathfrak{T}_{2n+1,q}\frac{\left(  -1\right)
^{k}t^{2n+1}}{\left[  2n+1\right]  _{q}!},
\end{align*}
and%
\[
\mathfrak{T}_{2n+1,q}=\mathfrak{g}_{2n+2,q}\frac{\left(  -1\right)
^{k-1}2^{2n+1}}{\left[  2n+2\right]  _{q}}.
\]

\end{proof}

The following result is $q$-analogue of the addition theorem for the classical
Bernoulli, Euler and Genocchi polynomials.

\begin{lemma}
\label{L:1}\emph{(Addition Theorems)} For all $x,y\in\mathbb{C}$ we have%
\begin{equation}%
\begin{tabular}
[c]{lll}%
$\mathfrak{B}_{n,q}\left(  x,y\right)  =%
{\displaystyle\sum\limits_{k=0}^{n}}
\left[
\begin{array}
[c]{c}%
n\\
k
\end{array}
\right]  _{q}\mathfrak{b}_{k,q}\left(  x\oplus_{q}y\right)  ^{n-k},$ &  &
$\mathfrak{B}_{n,q}\left(  x,y\right)  =%
{\displaystyle\sum\limits_{k=0}^{n}}
\left[
\begin{array}
[c]{c}%
n\\
k
\end{array}
\right]  _{q}\dfrac{(-1,q)_{n-k}}{2^{n-k}}\mathfrak{B}_{k,q}\left(  x\right)
y^{n-k},$\\
$\mathfrak{E}_{n,q}\left(  x,y\right)  =%
{\displaystyle\sum\limits_{k=0}^{n}}
\left[
\begin{array}
[c]{c}%
n\\
k
\end{array}
\right]  _{q}\mathfrak{e}_{k,q}\left(  x\oplus_{q}y\right)  ^{n-k},$ &  &
$\mathfrak{E}_{n,q}\left(  x,y\right)  =%
{\displaystyle\sum\limits_{k=0}^{n}}
\left[
\begin{array}
[c]{c}%
n\\
k
\end{array}
\right]  _{q}\dfrac{(-1,q)_{n-k}}{2^{n-k}}\mathfrak{E}_{k,q}\left(  x\right)
y^{n-k},$\\
$\mathfrak{G}_{n,q}\left(  x,y\right)  =%
{\displaystyle\sum\limits_{k=0}^{n}}
\left[
\begin{array}
[c]{c}%
n\\
k
\end{array}
\right]  _{q}\mathfrak{g}_{k,q}\left(  x\oplus_{q}y\right)  ^{n-k},$ &  &
$\mathfrak{G}_{n,q}\left(  x,y\right)  =%
{\displaystyle\sum\limits_{k=0}^{n}}
{\displaystyle\sum\limits_{k=0}^{n}}
\left[
\begin{array}
[c]{c}%
n\\
k
\end{array}
\right]  _{q}\dfrac{(-1,q)_{n-k}}{2^{n-k}}\mathfrak{G}_{k,q}\left(  x\right)
y^{n-k}.$%
\end{tabular}
\ \ \label{be01}%
\end{equation}

\end{lemma}

\begin{proof}
We prove only the first formula. It is a consequence of the following identity%
\begin{align*}
\sum_{n=0}^{\infty}\mathfrak{B}_{n,q}\left(  x,y\right)  \frac{t^{n}}{\left[
n\right]  _{q}!}  &  =\frac{t}{\mathcal{E}_{q}\left(  t\right)  -1}%
\mathcal{E}_{q}\left(  tx\right)  \mathcal{E}_{q}\left(  ty\right)
=\sum_{n=0}^{\infty}\mathfrak{b}_{n,q}\frac{t^{n}}{\left[  n\right]  _{q}%
!}\sum_{n=0}^{\infty}\left(  x\oplus_{q}y\right)  ^{n}\frac{t^{n}}{\left[
n\right]  _{q}!}\\
&  =\sum_{n=0}^{\infty}%
{\displaystyle\sum\limits_{k=0}^{n}}
\left[
\begin{array}
[c]{c}%
n\\
k
\end{array}
\right]  _{q}\mathfrak{b}_{k,q}\left(  x\oplus_{q}y\right)  ^{n-k}\frac{t^{n}%
}{\left[  n\right]  _{q}!}.
\end{align*}

\end{proof}

In particular, setting $y=0$ in (\ref{be01}), we get the following formulas
for $q$-Bernoulli, $q$-Euler and $q$-Genocchi polynomials, respectively.%
\begin{align}
\mathfrak{B}_{n,q}\left(  x\right)   &  =\sum_{k=0}^{n}\left[
\begin{array}
[c]{c}%
n\\
k
\end{array}
\right]  _{q}\dfrac{(-1,q)_{n-k}}{2^{n-k}}\mathfrak{b}_{k,q}x^{n-k}%
,\ \ \ \mathfrak{E}_{n,q}\left(  x\right)  =\sum_{k=0}^{n}\left[
\begin{array}
[c]{c}%
n\\
k
\end{array}
\right]  _{q}\dfrac{(-1,q)_{n-k}}{2^{n-k}}\mathfrak{e}_{k,q}x^{n-k}%
,\label{be7}\\
\mathfrak{G}_{n,q}\left(  x\right)   &  =\sum_{k=0}^{n}\left[
\begin{array}
[c]{c}%
n\\
k
\end{array}
\right]  _{q}\dfrac{(-1,q)_{n-k}}{2^{n-k}}\mathfrak{g}_{k,q}x^{n-k}.
\label{be8}%
\end{align}
Setting $y=1$ in (\ref{be01}), we get%
\begin{align}
\mathfrak{B}_{n,q}\left(  x,1\right)   &  =\sum_{k=0}^{n}\left[
\begin{array}
[c]{c}%
n\\
k
\end{array}
\right]  _{q}\dfrac{(-1,q)_{n-k}}{2^{n-k}}\mathfrak{B}_{k,q}\left(  x\right)
,\ \ \ \mathfrak{E}_{n,q}\left(  x,1\right)  =\sum_{k=0}^{n}\left[
\begin{array}
[c]{c}%
n\\
k
\end{array}
\right]  _{q}\dfrac{(-1,q)_{n-k}}{2^{n-k}}\mathfrak{E}_{k,q}\left(  x\right)
,\label{be3}\\
\mathfrak{G}_{n,q}\left(  x,1\right)   &  =\sum_{k=0}^{n}\left[
\begin{array}
[c]{c}%
n\\
k
\end{array}
\right]  _{q}\dfrac{(-1,q)_{n-k}}{2^{n-k}}\mathfrak{G}_{k,q}\left(  x\right)
. \label{be4}%
\end{align}
Clearly (\ref{be3}) and (\ref{be4}) are $q$-analogues of%
\[
B_{n}\left(  x+1\right)  =\sum_{k=0}^{n}\left(
\begin{array}
[c]{c}%
n\\
k
\end{array}
\right)  B_{k}\left(  x\right)  ,\ E_{n}\left(  x+1\right)  =\sum_{k=0}%
^{n}\left(
\begin{array}
[c]{c}%
n\\
k
\end{array}
\right)  E_{k}\left(  x\right)  ,\ G_{n}\left(  x+1\right)  =\sum_{k=0}%
^{n}\left(
\begin{array}
[c]{c}%
n\\
k
\end{array}
\right)  G_{k}\left(  x\right)  ,
\]
respectively.

\begin{lemma}
The odd coefficient of the $q$-Bernoulli numbers except the first one are
zero, that means $\mathfrak{b}_{n,q}=0$ where $n=2r+1,\ (r\in%
\mathbb{N}
)$.

\begin{proof}
It follows from the fact that the function
\[
f(t)=%
{\displaystyle\sum\limits_{n=0}^{\infty}}
\mathfrak{b}_{n,q}\frac{t^{n}}{[n]_{q}!}-\mathfrak{b}_{1,q}t=\frac
{t}{\mathcal{E}_{q}\left(  t\right)  -1}+\ \frac{t}{2}=\ \frac{t}{2}\left(
\frac{\mathcal{E}_{q}\left(  t\right)  +1}{\mathcal{E}_{q}\left(  t\right)
-1}\right)  \ \
\]
is even, and the coefficient of $t^{n}$ in the Taylor expansion about zero of
any even function vanish for all odd $n$. note that this could not happen in
the last $q$-analogue of these numbers, because in the case of improved
exponential function $\mathcal{E}_{q}\left(  -t\right)  =\left(
\mathcal{E}_{q}\left(  t\right)  \right)  ^{-1}$.
\end{proof}
\end{lemma}

By using (\ref{be07}) and $q$-derivative approaching to the next lemma.

\begin{lemma}
\label{L:2}We have%
\begin{align*}
D_{q,x}\mathfrak{B}_{n,q}\left(  x\right)    & =\left[  n\right]  _{q}%
\frac{\mathfrak{B}_{n-1,q}\left(  x\right)  +\mathfrak{B}_{n-1,q}\left(
qx\right)  }{2},\ \ \ D_{q,x}\mathfrak{E}_{n,q}\left(  x\right)  =\left[
n\right]  _{q}\frac{\mathfrak{E}_{n-1,q}\left(  x\right)  +\mathfrak{E}%
_{n-1,q}\left(  qx\right)  }{2},\ \ \ \\
D_{q,x}\mathfrak{G}_{n,q}\left(  x\right)    & =\left[  n\right]  _{q}%
\frac{\mathfrak{G}_{n-1,q}\left(  x\right)  +\mathfrak{G}_{n-1,q}\left(
qx\right)  }{2}.
\end{align*}

\end{lemma}

\begin{lemma}
\label{L:3}\emph{(Difference Equations)} We have%
\begin{align}
\mathfrak{B}_{n,q}\left(  x,1\right)  -\mathfrak{B}_{n,q}\left(  x\right)   &
=\frac{\left(  -1;q\right)  _{n-1}}{2^{n-1}}\left[  n\right]  _{q}%
x^{n-1},\ \ n\geq1,\label{be5}\\
\mathfrak{E}_{n,q}\left(  x,1\right)  +\mathfrak{E}_{n,q}\left(  x\right)   &
=2\frac{\left(  -1;q\right)  _{n}}{2^{n}}x^{n},\ \ \ n\geq0,\label{be6}\\
\mathfrak{G}_{n,q}\left(  x,1\right)  +\mathfrak{G}_{n,q}\left(  x\right)   &
=2\frac{\left(  -1;q\right)  _{n-1}}{2^{n-1}}\left[  n\right]  _{q}%
x^{n-1},\ \ \ n\geq1. \label{b61}%
\end{align}

\end{lemma}

\begin{proof}
We prove the identity for the $q$-Bernoulli polynomials. From the identity%
\[
\frac{t\mathcal{E}_{q}\left(  t\right)  }{\mathcal{E}_{q}\left(  t\right)
-1}\mathcal{E}_{q}\left(  tx\right)  =t\mathcal{E}_{q}\left(  tx\right)
+\frac{t}{\mathcal{E}_{q}\left(  t\right)  -1}\mathcal{E}_{q}\left(
tx\right)  ,
\]
it follows that%
\[
\sum_{n=0}^{\infty}%
{\displaystyle\sum\limits_{k=0}^{n}}
\left[
\begin{array}
[c]{c}%
n\\
k
\end{array}
\right]  _{q}\frac{(-1,q)_{n-k}}{2^{n-k}}\mathfrak{B}_{k,q}\left(  x\right)
\frac{t^{n}}{\left[  n\right]  _{q}!}=\sum_{n=0}^{\infty}\frac{(-1,q)_{n}%
}{2^{n}}x^{n}\frac{t^{n+1}}{\left[  n\right]  _{q}!}+\sum_{n=0}^{\infty
}\mathfrak{B}_{n,q}\left(  x\right)  \frac{t^{n}}{\left[  n\right]  _{q}!}.
\]

\end{proof}

From (\ref{be5}) and (\ref{be7}), (\ref{be6}) and (\ref{be8}) we obtain the
following formulas.

\begin{lemma}
\label{L:4}We have
\begin{align}
x^{n}  &  =\frac{2^{n}}{\left(  -1;q\right)  _{n}\left[  n\right]  _{q}}%
\sum_{k=0}^{n}\left[
\begin{array}
[c]{c}%
n+1\\
k
\end{array}
\right]  _{q}\frac{\left(  -1;q\right)  _{n+1-k}}{2^{n+1-k}}\mathfrak{B}%
_{k,q}\left(  x\right) \label{be9}\\
x^{n}  &  =\frac{2^{n-1}}{\left(  -1;q\right)  _{n}}\left(  \sum_{k=0}%
^{n}\left[
\begin{array}
[c]{c}%
n\\
k
\end{array}
\right]  _{q}\frac{\left(  -1;q\right)  _{n-k}}{2^{n-k}}\mathfrak{E}%
_{k,q}\left(  x\right)  +\mathfrak{E}_{n,q}\left(  x\right)  \right)  ,\\
x^{n}  &  =\frac{2^{n-1}}{\left(  -1;q\right)  _{n}\left[  n+1\right]  _{q}%
}\left(  \sum_{k=0}^{n+1}\left[
\begin{array}
[c]{c}%
n+1\\
k
\end{array}
\right]  _{q}\frac{\left(  -1;q\right)  _{n+1-k}}{2^{n+1-k}}\mathfrak{G}%
_{k,q}\left(  x\right)  +\mathfrak{G}_{n+1,q}\left(  x\right)  \right)  .
\label{be10}%
\end{align}

\end{lemma}

The above formulas are $q$-analoques of the following familiar expansions%
\begin{align}
x^{n}  & =\frac{1}{n+1}%
{\displaystyle\sum\limits_{k=0}^{n}}
\left(
\begin{array}
[c]{c}%
n+1\\
k
\end{array}
\right)  B_{k}\left(  x\right)  ,\ \ \ x^{n}=\frac{1}{2}\left[
{\displaystyle\sum\limits_{k=0}^{n}}
\left(
\begin{array}
[c]{c}%
n\\
k
\end{array}
\right)  E_{k}\left(  x\right)  +E_{n}\left(  x\right)  \right]
,\ \ \ \label{cl1}\\
x^{n}  & =\frac{1}{2\left(  n+1\right)  }\left[
{\displaystyle\sum\limits_{k=0}^{n+1}}
\left(
\begin{array}
[c]{c}%
n+1\\
k
\end{array}
\right)  E_{k}\left(  x\right)  +E_{n+1}\left(  x\right)  \right]  ,\nonumber
\end{align}
respectively.

\begin{lemma}
The following identities hold true.%
\begin{align*}%
{\displaystyle\sum\limits_{k=0}^{n}}
\left[
\begin{array}
[c]{c}%
n\\
k
\end{array}
\right]  _{q}\frac{(-1,q)_{n-k}}{2^{n-k}}\mathfrak{B}_{k,q}\left(  x,y\right)
-\mathfrak{B}_{n,q}\left(  x,y\right)   &  =\left[  n\right]  _{q}\left(
x\oplus_{q}y\right)  ^{n-1},\\%
{\displaystyle\sum\limits_{k=0}^{n}}
\left[
\begin{array}
[c]{c}%
n\\
k
\end{array}
\right]  _{q}\frac{(-1,q)_{n-k}}{2^{n-k}}\mathfrak{E}_{k,q}\left(  x,y\right)
+\mathfrak{E}_{n,q}\left(  x,y\right)   &  =2\left(  x\oplus_{q}y\right)
^{n},\\%
{\displaystyle\sum\limits_{k=0}^{n}}
\left[
\begin{array}
[c]{c}%
n\\
k
\end{array}
\right]  _{q}\frac{(-1,q)_{n-k}}{2^{n-k}}\mathfrak{G}_{k,q}\left(  x,y\right)
+\mathfrak{G}_{n,q}\left(  x,y\right)   &  =2\left[  n\right]  _{q}\left(
x\oplus_{q}y\right)  ^{n-1}.
\end{align*}

\end{lemma}

\begin{proof}
We the identity for the $q$-Bernoulli polynomials. From the identity%
\[
\frac{t\mathcal{E}_{q}\left(  t\right)  }{\mathcal{E}_{q}\left(  t\right)
-1}\mathcal{E}_{q}\left(  tx\right)  \mathcal{E}_{q}\left(  ty\right)
=t\mathcal{E}_{q}\left(  tx\right)  \mathcal{E}_{q}\left(  ty\right)
+\frac{t}{\mathcal{E}_{q}\left(  t\right)  -1}\mathcal{E}_{q}\left(
tx\right)  \mathcal{E}_{q}\left(  ty\right)  ,
\]
it follows that%
\[
\sum_{n=0}^{\infty}%
{\displaystyle\sum\limits_{k=0}^{n}}
\left[
\begin{array}
[c]{c}%
n\\
k
\end{array}
\right]  _{q}\frac{(-1,q)_{n-k}}{2^{n-k}}\mathfrak{B}_{k,q}\left(  x,y\right)
\frac{t^{n}}{\left[  n\right]  _{q}!}=\sum_{n=0}^{\infty}%
{\displaystyle\sum\limits_{k=0}^{n}}
\left[
\begin{array}
[c]{c}%
n\\
k
\end{array}
\right]  _{q}\frac{(-1,q)_{k}(-1,q)_{n-k}}{2^{n}}x^{k}y^{n-k}\frac{t^{n+1}%
}{\left[  n\right]  _{q}!}+\sum_{n=0}^{\infty}\mathfrak{B}_{n,q}\left(
x,y\right)  \frac{t^{n}}{\left[  n\right]  _{q}!}.
\]

\end{proof}

\section{Some new formulas}

The classical Cayley transformation $z\rightarrow$Cay$(z,a):=\frac{1+az}%
{1-az}$ motivates us to approaching to the formula for $\mathcal{E}_{q}\left(
qt\right)  $, In addition by substitute it in the generating formula we have:%

\[
\widehat{\mathfrak{B}}_{q}\left(  qt\right)  \widehat{\mathfrak{B}}_{q}\left(
t\right)  =\left(  \widehat{\mathfrak{B}}_{q}\left(  qt\right)  -q\widehat
{\mathfrak{B}}_{q}\left(  t\right)  (1+(1-q)\frac{t}{2})\right)  \frac{1}%
{1-q}\times\frac{2}{\mathcal{E}_{q}\left(  t\right)  +1}%
\]
The right hand side can be presented by improved $q$-Euler numbers .Now the
equating coefficients of $t^{n}$ we get the following identity.In the case
that $n=0$, we find the first improved $q$-Euler number which is exactly 1.

\begin{proposition}
For all $n\geq1$,%

\[%
{\displaystyle\sum\limits_{k=0}^{n}}
\left[
\begin{array}
[c]{c}%
n\\
k
\end{array}
\right]  _{q}\mathfrak{B}_{k,q}\mathfrak{B}_{n-k,q}q^{k}=-q%
{\displaystyle\sum\limits_{k=0}^{n}}
\left[
\begin{array}
[c]{c}%
n\\
k
\end{array}
\right]  _{q}\mathfrak{B}_{k,q}\mathfrak{E}_{n-k,q}[k-1]_{q}-\frac{q}{2}%
{\displaystyle\sum\limits_{k=0}^{n-1}}
\left[
\begin{array}
[c]{c}%
n-1\\
k
\end{array}
\right]  _{q}\mathfrak{B}_{k,q}\mathfrak{E}_{n-k-1,q}[n]_{q}%
\]

\end{proposition}

Let take a $q$-derivative from generating function, after simplifying the
equation and by knowing the quotient rule for quantum derivative , also using
that
\[
\mathcal{E}_{q}\left(  qt\right)  =\frac{1-(1-q)\frac{t}{2}}{1+(1-q)\frac
{t}{2}}\mathcal{E}_{q}\left(  t\right)  ,D_{q}(\mathcal{E}_{q}\left(
t\right)  )=\frac{\mathcal{E}_{q}\left(  qt\right)  +\mathcal{E}_{q}\left(
t\right)  }{2},
\]
we have:%
\[
\widehat{B}_{q}\left(  qt\right)  \widehat{B}_{q}\left(  t\right)
=\frac{2+(1-q)t}{2\mathcal{E}_{q}\left(  t\right)  (q-1)}\left(  q\widehat
{B}_{q}\left(  t\right)  -\widehat{B}_{q}\left(  qt\right)  \right)
\]

It is clear that $\mathcal{E}_{q}^{-1}\left(  t\right)  =\mathcal{E}%
_{q}\left(  -t\right)  $. Now the equating coefficient of $t^{n}$ we lead to
the following identity.

\begin{proposition}
For all $n\geq1$,%

\[%
{\displaystyle\sum\limits_{k=0}^{2n}}
\left[
\begin{array}
[c]{c}%
2n\\
k
\end{array}
\right]  _{q}\mathfrak{B}_{k,q}\mathfrak{B}_{2n-k,q}q^{k}=-q%
{\displaystyle\sum\limits_{k=0}^{2n}}
\left\{
\begin{array}
[c]{c}%
2n\\
k
\end{array}
\right\}  _{q}\mathfrak{B}_{k,q}[k-1]_{q}(-1)^{k}+\frac{q(1-q)}{2}%
{\displaystyle\sum\limits_{k=0}^{2n-1}}
\left\{
\begin{array}
[c]{c}%
2n-1\\
k
\end{array}
\right\}  _{q}\mathfrak{B}_{k,q}[k-1]_{q}(-1)^{k}%
\]

\[%
{\displaystyle\sum\limits_{k=0}^{2n+1}}
\left[
\begin{array}
[c]{c}%
2n+1\\
k
\end{array}
\right]  _{q}\mathfrak{B}_{k,q}\mathfrak{B}_{2n-k+1,q}q^{k}=q%
{\displaystyle\sum\limits_{k=0}^{2n+1}}
\left\{
\begin{array}
[c]{c}%
2n+1\\
k
\end{array}
\right\}  _{q}\mathfrak{B}_{k,q}[k-1]_{q}(-1)^{k}-\frac{q(1-q)}{2}%
{\displaystyle\sum\limits_{k=0}^{2n}}
\left\{
\begin{array}
[c]{c}%
2n\\
k
\end{array}
\right\}  _{q}\mathfrak{B}_{k,q}[k-1]_{q}(-1)^{k}%
\]

\end{proposition}

We may also derive a differential equation for $\widehat{B}_{q}\left(
t\right)  .$If we differentiate both sides of generating function with respect
to $t$, after a little calculation we find that%

\[
\frac{\partial}{\partial t}\widehat{B}_{q}\left(  t\right)  =\widehat{B}%
_{q}\left(  t\right)  \left(  \frac{1}{t}-\frac{(1-q)\mathcal{E}_{q}\left(
t\right)  }{\mathcal{E}_{q}\left(  t\right)  -1}\left(
{\displaystyle\sum\limits_{k=0}^{\infty}}
\frac{4q^{k}}{4-(1-q)^{2}q^{2k}}\right)  \right)
\]

If we differentiate with respect to $q$, we instead obtain%

\[
\frac{\partial}{\partial q}\widehat{B}_{q}\left(  t\right)  =-\widehat{B}%
_{q}^{2}\left(  t\right)  \mathcal{E}_{q}\left(  t\right)
{\displaystyle\sum\limits_{k=0}^{\infty}}
\frac{4t(kq^{k-1}-(k+1)q^{k})}{4-(1-q)^{2}q^{2k}}%
\]

Again using generating function and combining this with the t derivative we
get the partial differential equation

\begin{proposition}

\[
\frac{\partial}{\partial t}\widehat{B}_{q}\left(  t\right)  -\frac{\partial
}{\partial q}\widehat{B}_{q}\left(  t\right)  =\frac{\widehat{B}_{q}\left(
t\right)  }{t}+\frac{\widehat{B}_{q}^{2}\left(  t\right)  \mathcal{E}%
_{q}\left(  t\right)  }{t}%
{\displaystyle\sum\limits_{k=0}^{\infty}}
\frac{4t(kq^{k-1}-(k+1)q^{k})-q^{k}(1-q)}{4-(1-q)^{2}q^{2k}}%
\]

\end{proposition}

\section{Explicit relationship between the $q$-Bernoulli and $q$-Euler
polynomials}

In this section we will give some explicit relationships between the
$q$-Bernoulli and $q$-Euler polynomials. Here some $q$-analogues of known
results will be given. We also obtain new formulas and their some special
cases below. These formulas are some extensions of the formulas of Srivastava
and Pint\'{e}r, Cheon and others.

We present natural $q$-extensions of the main results in the papers
\cite{pinter} and \cite{luo2}, see Theorems \ref{S-P1} and \ref{S-P2}.

\begin{theorem}
\label{S-P1}For $n\in\mathbb{N}_{0}$, the following relationships hold true:%
\begin{align*}
\mathfrak{B}_{n,q}\left(  x,y\right)   &  =\frac{1}{2}%
{\displaystyle\sum\limits_{k=0}^{n}}
\left[
\begin{array}
[c]{c}%
n\\
k
\end{array}
\right]  _{q}m^{k-n}\left[  \mathfrak{B}_{k,q}\left(  x\right)  +%
{\displaystyle\sum\limits_{j=0}^{k}}
\left\{
\begin{array}
[c]{c}%
k\\
j
\end{array}
\right\}  _{q}\frac{\mathfrak{B}_{j,q}\left(  x\right)  }{m^{k-j}}\right]
\mathfrak{E}_{n-k,q}\left(  my\right) \\
&  =\frac{1}{2}%
{\displaystyle\sum\limits_{k=0}^{n}}
\left[
\begin{array}
[c]{c}%
n\\
k
\end{array}
\right]  _{q}m^{k-n}\left[  \mathfrak{B}_{k,q}\left(  x\right)  +\mathfrak{B}%
_{k,q}\left(  x,\frac{1}{m}\right)  \right]  \mathfrak{E}_{n-k,q}\left(
my\right)  .
\end{align*}

\end{theorem}

\begin{proof}
Using the following identity%
\[
\frac{t}{\mathcal{E}_{q}\left(  t\right)  -1}\mathcal{E}_{q}\left(  tx\right)
\mathcal{E}_{q}\left(  ty\right)  =\frac{t}{\mathcal{E}_{q}\left(  t\right)
-1}\mathcal{E}_{q}\left(  tx\right)  \cdot\frac{\mathcal{E}_{q}\left(
\frac{t}{m}\right)  +1}{2}\cdot\frac{2}{\mathcal{E}_{q}\left(  \frac{t}%
{m}\right)  +1}\mathcal{E}_{q}\left(  \frac{t}{m}my\right)
\]
we have%
\begin{align*}%
{\displaystyle\sum\limits_{n=0}^{\infty}}
\mathfrak{B}_{n,q}\left(  x,y\right)  \frac{t^{n}}{\left[  n\right]  _{q}!}
&  =\frac{1}{2}%
{\displaystyle\sum\limits_{n=0}^{\infty}}
\mathfrak{E}_{n,q}\left(  my\right)  \frac{t^{n}}{m^{n}\left[  n\right]
_{q}!}%
{\displaystyle\sum\limits_{n=0}^{\infty}}
\frac{(-1,q)_{n}}{m^{n}2^{n}}\frac{t^{n}}{\left[  n\right]  _{q}!}%
{\displaystyle\sum\limits_{n=0}^{\infty}}
\mathfrak{B}_{n,q}\left(  x\right)  \frac{t^{n}}{\left[  n\right]  _{q}!}\\
&  +\frac{1}{2}%
{\displaystyle\sum\limits_{n=0}^{\infty}}
\mathfrak{E}_{n,q}\left(  my\right)  \frac{t^{n}}{m^{n}\left[  n\right]
_{q}!}%
{\displaystyle\sum\limits_{n=0}^{\infty}}
\mathfrak{B}_{n,q}\left(  x\right)  \frac{t^{n}}{\left[  n\right]  _{q}!}\\
&  =:I_{1}+I_{2}.
\end{align*}
It is clear that%
\[
I_{2}=\frac{1}{2}%
{\displaystyle\sum\limits_{n=0}^{\infty}}
\mathfrak{E}_{n,q}\left(  my\right)  \frac{t^{n}}{m^{n}\left[  n\right]
_{q}!}%
{\displaystyle\sum\limits_{n=0}^{\infty}}
\mathfrak{B}_{n,q}\left(  x\right)  \frac{t^{n}}{\left[  n\right]  _{q}%
!}=\frac{1}{2}%
{\displaystyle\sum\limits_{n=0}^{\infty}}
{\displaystyle\sum\limits_{k=0}^{n}}
\left[
\begin{array}
[c]{c}%
n\\
j
\end{array}
\right]  _{q}m^{k-n}\mathfrak{B}_{k,q}\left(  x\right)  \mathfrak{E}%
_{n-k,q}\left(  my\right)  \frac{t^{n}}{\left[  n\right]  _{q}!}.
\]
On the other hand%
\begin{align*}
I_{1}  &  =\frac{1}{2}%
{\displaystyle\sum\limits_{n=0}^{\infty}}
\mathfrak{E}_{n,q}\left(  my\right)  \frac{t^{n}}{m^{n}\left[  n\right]
_{q}!}%
{\displaystyle\sum\limits_{n=0}^{\infty}}
{\displaystyle\sum\limits_{j=0}^{n}}
\left\{
\begin{array}
[c]{c}%
n\\
j
\end{array}
\right\}  _{q}\mathfrak{B}_{j,q}\left(  x\right)  \frac{t^{n}}{m^{n-j}\left[
n\right]  _{q}!}\\
&  =\frac{1}{2}%
{\displaystyle\sum\limits_{n=0}^{\infty}}
{\displaystyle\sum\limits_{k=0}^{n}}
\left[
\begin{array}
[c]{c}%
n\\
k
\end{array}
\right]  _{q}\mathfrak{E}_{n-k,q}\left(  my\right)
{\displaystyle\sum\limits_{j=0}^{k}}
\left\{
\begin{array}
[c]{c}%
k\\
j
\end{array}
\right\}  _{q}\frac{\mathfrak{B}_{j,q}\left(  x\right)  }{m^{n-k}m^{k-j}}%
\frac{t^{n}}{\left[  n\right]  _{q}!}.
\end{align*}
Therefore%
\[%
{\displaystyle\sum\limits_{n=0}^{\infty}}
\mathfrak{B}_{n,q}\left(  x,y\right)  \frac{t^{n}}{\left[  n\right]  _{q}%
!}=\frac{1}{2}%
{\displaystyle\sum\limits_{n=0}^{\infty}}
{\displaystyle\sum\limits_{k=0}^{n}}
\left[
\begin{array}
[c]{c}%
n\\
k
\end{array}
\right]  _{q}m^{k-n}\left[  \mathfrak{B}_{k,q}\left(  x\right)  +%
{\displaystyle\sum\limits_{j=0}^{k}}
\left\{
\begin{array}
[c]{c}%
k\\
j
\end{array}
\right\}  _{q}\frac{\mathfrak{B}_{j,q}\left(  x\right)  }{m^{k-j}}\right]
\mathfrak{E}_{n-k,q}\left(  my\right)  \frac{t^{n}}{\left[  n\right]  _{q}!}.
\]
It remains to equate coefficient of $t^{n}.$
\end{proof}

Next we discuss some special cases of Theorem \ref{S-P1}.

\begin{corollary}
\label{C:3}For $n\in\mathbb{N}_{0}$ the following relationship holds true.%
\begin{equation}
\mathfrak{B}_{n,q}\left(  x,y\right)  =%
{\displaystyle\sum\limits_{k=0}^{n}}
\left[
\begin{array}
[c]{c}%
n\\
k
\end{array}
\right]  _{q}\left(  \mathfrak{B}_{k,q}\left(  x\right)  +\frac{\left(
-1;q\right)  _{k-1}}{2^{k}}\left[  k\right]  _{q}x^{k-1}\right)
\mathfrak{E}_{n-k,q}\left(  y\right)  . \label{cw1}%
\end{equation}

\end{corollary}

The formula (\ref{cw1}) ia a $q$-extension of the Cheon's main result
\cite{cheon}.

\begin{theorem}
\label{S-P2} For $n\in\mathbb{N}_{0}$, the following relationships%
\[
\mathfrak{E}_{n,q}\left(  x,y\right)  =\frac{1}{\left[  n+1\right]  _{q}}%
{\displaystyle\sum\limits_{k=0}^{n+1}}
\frac{1}{m^{n+1-k}}\left[
\begin{array}
[c]{c}%
n+1\\
k
\end{array}
\right]  _{q}\left(
{\displaystyle\sum\limits_{j=0}^{k}}
\left\{
\begin{array}
[c]{c}%
k\\
j
\end{array}
\right\}  _{q}\frac{\mathfrak{E}_{j,q}\left(  x\right)  }{m^{k-j}%
}-\mathfrak{E}_{k,q}\left(  y\right)  \right)  \mathfrak{B}_{n+1-k,q}\left(
mx\right)
\]
hold true between the $q$-Bernoulli polynomials and $q$-Euler polynomials.
\end{theorem}

\begin{proof}
The proof is based on the following identity%
\[
\frac{2}{\mathcal{E}_{q}\left(  t\right)  +1}\mathcal{E}_{q}\left(  tx\right)
\mathcal{E}_{q}\left(  ty\right)  =\frac{2}{\mathcal{E}_{q}\left(  t\right)
+1}\mathcal{E}_{q}\left(  ty\right)  \cdot\frac{\mathcal{E}_{q}\left(
\frac{t}{m}\right)  -1}{t}\cdot\frac{t}{\mathcal{E}_{q}\left(  \frac{t}%
{m}\right)  -1}\mathcal{E}_{q}\left(  \frac{t}{m}mx\right)  .
\]
Indeed%
\begin{align*}%
{\displaystyle\sum\limits_{n=0}^{\infty}}
\mathfrak{E}_{n,q}\left(  x,y\right)  \frac{t^{n}}{\left[  n\right]  _{q}!}
&  =%
{\displaystyle\sum\limits_{n=0}^{\infty}}
\mathfrak{E}_{n,q}\left(  y\right)  \frac{t^{n}}{\left[  n\right]  _{q}!}%
{\displaystyle\sum\limits_{n=0}^{\infty}}
\frac{t^{n-1}}{m^{n}\left\{  n\right\}  _{q}!}%
{\displaystyle\sum\limits_{n=0}^{\infty}}
\mathfrak{B}_{n,q}\left(  mx\right)  \frac{t^{n}}{m^{n}\left[  n\right]
_{q}!}\\
&  -%
{\displaystyle\sum\limits_{n=0}^{\infty}}
\mathfrak{E}_{n,q}\left(  y\right)  \frac{t^{n-1}}{\left[  n\right]  _{q}!}%
{\displaystyle\sum\limits_{n=0}^{\infty}}
\mathfrak{B}_{n,q}\left(  mx\right)  \frac{t^{n}}{m^{n}\left[  n\right]
_{q}!}\\
&  =:I_{1}-I_{2}.
\end{align*}
It follows that%
\begin{align*}
I_{2}  &  =\dfrac{1}{t}%
{\displaystyle\sum\limits_{n=0}^{\infty}}
\mathfrak{E}_{n,q}\left(  y\right)  \frac{t^{n}}{\left[  n\right]  _{q}!}%
{\displaystyle\sum\limits_{n=0}^{\infty}}
\mathfrak{B}_{n,q}\left(  mx\right)  \frac{t^{n}}{m^{n}\left[  n\right]
_{q}!}=\dfrac{1}{t}%
{\displaystyle\sum\limits_{n=0}^{\infty}}
{\displaystyle\sum\limits_{k=0}^{n}}
\left[
\begin{array}
[c]{c}%
n\\
k
\end{array}
\right]  _{q}\frac{1}{m^{n-k}}\mathfrak{E}_{k,q}\left(  y\right)
\mathfrak{B}_{n-k,q}\left(  mx\right)  \frac{t^{n}}{\left[  n\right]  _{q}!}\\
&  =%
{\displaystyle\sum\limits_{n=0}^{\infty}}
\frac{1}{\left[  n+1\right]  _{q}}%
{\displaystyle\sum\limits_{k=0}^{n+1}}
\left[
\begin{array}
[c]{c}%
n+1\\
k
\end{array}
\right]  _{q}\frac{1}{m^{n+1-k}}\mathfrak{E}_{k,q}\left(  y\right)
\mathfrak{B}_{n+1-k,q}\left(  mx\right)  \frac{t^{n}}{\left[  n\right]  _{q}%
!},
\end{align*}
and%
\begin{align*}
I_{1}  &  =\dfrac{1}{t}%
{\displaystyle\sum\limits_{n=0}^{\infty}}
\mathfrak{B}_{n,q}\left(  mx\right)  \frac{t^{n}}{m^{n}\left[  n\right]
_{q}!}%
{\displaystyle\sum\limits_{n=0}^{\infty}}
{\displaystyle\sum\limits_{k=0}^{n}}
\left\{
\begin{array}
[c]{c}%
n\\
k
\end{array}
\right\}  _{q}\frac{\mathfrak{E}_{k,q}\left(  y\right)  }{m^{n-k}}\frac{t^{n}%
}{\left[  n\right]  _{q}!}\\
&  =\dfrac{1}{t}%
{\displaystyle\sum\limits_{n=0}^{\infty}}
{\displaystyle\sum\limits_{k=0}^{n}}
\left[
\begin{array}
[c]{c}%
n\\
k
\end{array}
\right]  _{q}\mathfrak{B}_{n-k,q}\left(  mx\right)
{\displaystyle\sum\limits_{j=0}^{k}}
\left\{
\begin{array}
[c]{c}%
k\\
j
\end{array}
\right\}  _{q}\frac{\mathfrak{E}_{j,q}\left(  y\right)  }{m^{n-k}m^{k-j}}%
\frac{t^{n}}{\left[  n\right]  _{q}!}\\
&  =%
{\displaystyle\sum\limits_{n=0}^{\infty}}
{\displaystyle\sum\limits_{k=0}^{n}}
\left[
\begin{array}
[c]{c}%
n\\
j
\end{array}
\right]  _{q}\mathfrak{B}_{n-j,q}\left(  mx\right)
{\displaystyle\sum\limits_{k=0}^{j}}
\left\{
\begin{array}
[c]{c}%
j\\
k
\end{array}
\right\}  _{q}\frac{\mathfrak{E}_{k,q}\left(  x\right)  }{m^{n-k}}%
\frac{t^{n-1}}{\left[  n\right]  _{q}!}\\
&  =%
{\displaystyle\sum\limits_{n=0}^{\infty}}
\frac{1}{\left[  n+1\right]  _{q}}%
{\displaystyle\sum\limits_{j=0}^{n+1}}
\left[
\begin{array}
[c]{c}%
n+1\\
j
\end{array}
\right]  _{q}\mathfrak{B}_{n+1-j,q}\left(  mx\right)
{\displaystyle\sum\limits_{k=0}^{j}}
\left\{
\begin{array}
[c]{c}%
j\\
k
\end{array}
\right\}  _{q}\frac{\mathfrak{E}_{k,q}\left(  x\right)  }{m^{n+1-k}}%
\frac{t^{n}}{\left[  n\right]  _{q}!}.
\end{align*}

\end{proof}

Next we give an interesting relationship between the $q$-Genocchi polynomials
and the $q$-Bernoulli polynomials.

\begin{theorem}
\label{S-P3}For $n\in\mathbb{N}_{0}$, the following relationship%
\begin{align*}
\mathfrak{G}_{n,q}\left(  x,y\right)   &  =\dfrac{1}{\left[  n+1\right]  _{q}}%
{\displaystyle\sum\limits_{k=0}^{n+1}}
\frac{1}{m^{n-k}}\left[
\begin{array}
[c]{c}%
n+1\\
k
\end{array}
\right]  _{q}\left(
{\displaystyle\sum\limits_{j=0}^{k}}
\left[
\begin{array}
[c]{c}%
k\\
j
\end{array}
\right]  _{q}\frac{(-1,q)_{k-j}}{m^{k-j}2^{k-j}}\mathfrak{G}_{j,q}\left(
x\right)  -\mathfrak{G}_{k,q}\left(  x\right)  \right)  \mathfrak{B}%
_{n+1-k,q}\left(  my\right)  ,\\
\mathfrak{B}_{n,q}\left(  x,y\right)   &  =\dfrac{1}{2\left[  n+1\right]
_{q}}%
{\displaystyle\sum\limits_{k=0}^{n+1}}
\frac{1}{m^{n-k}}\left[
\begin{array}
[c]{c}%
n+1\\
k
\end{array}
\right]  _{q}\left(
{\displaystyle\sum\limits_{j=0}^{k}}
\left[
\begin{array}
[c]{c}%
k\\
j
\end{array}
\right]  _{q}\frac{(-1,q)_{k-j}}{m^{k-j}2^{k-j}}\mathfrak{B}_{j,q}\left(
x\right)  +\mathfrak{B}_{k,q}\left(  x\right)  \right)  \mathfrak{G}%
_{n+1-k,q}\left(  my\right)
\end{align*}
holds true between the $q$-Genocchi and the $q$-Bernoulli polynomials.
\end{theorem}

\begin{proof}
Using the following identity%
\begin{align*}
&  \frac{2t}{\mathcal{E}_{q}\left(  t\right)  +1}\mathcal{E}_{q}\left(
tx\right)  \mathcal{E}_{q}\left(  ty\right)  \\
&  =\frac{2t}{\mathcal{E}_{q}\left(  t\right)  +1}\mathcal{E}_{q}\left(
tx\right)  \cdot\left(  \mathcal{E}_{q}\left(  \frac{t}{m}\right)  -1\right)
\frac{m}{t}\cdot\frac{\frac{t}{m}}{\mathcal{E}_{q}\left(  \frac{t}{m}\right)
-1}\cdot\mathcal{E}_{q}\left(  \frac{t}{m}my\right)
\end{align*}
we have%
\begin{align*}
&
{\displaystyle\sum\limits_{n=0}^{\infty}}
\mathfrak{G}_{n,q}\left(  x,y\right)  \frac{t^{n}}{\left[  n\right]  _{q}!}\\
&  =\frac{m}{t}%
{\displaystyle\sum\limits_{n=0}^{\infty}}
\mathfrak{G}_{n,q}\left(  x,y\right)  \frac{t^{n}}{\left[  n\right]  _{q}!}%
{\displaystyle\sum\limits_{n=0}^{\infty}}
\frac{(-1,q)_{n}}{m^{n}2^{n}}\frac{t^{n}}{\left[  n\right]  _{q}!}%
{\displaystyle\sum\limits_{n=0}^{\infty}}
\mathfrak{B}_{n,q}\left(  my\right)  \frac{t^{n}}{m^{n}\left[  n\right]
_{q}!}\\
&  -\frac{m}{t}%
{\displaystyle\sum\limits_{n=0}^{\infty}}
\mathfrak{G}_{n,q}\left(  x,y\right)  \frac{t^{n}}{\left[  n\right]  _{q}!}%
{\displaystyle\sum\limits_{n=0}^{\infty}}
\mathfrak{B}_{n,q}\left(  my\right)  \frac{t^{n}}{m^{n}\left[  n\right]
_{q}!}\\
&  =\frac{m}{t}%
{\displaystyle\sum\limits_{n=0}^{\infty}}
\left(
{\displaystyle\sum\limits_{k=0}^{n}}
\left[
\begin{array}
[c]{c}%
n\\
k
\end{array}
\right]  _{q}\frac{(-1,q)_{n-k}}{m^{n-k}2^{n-k}}\mathfrak{G}_{k,q}\left(
x\right)  -\mathfrak{G}_{n,q}\left(  x\right)  \right)  \frac{t^{n}}{\left[
n\right]  _{q}!}%
{\displaystyle\sum\limits_{n=0}^{\infty}}
\mathfrak{B}_{n,q}\left(  my\right)  \frac{t^{n}}{m^{n}\left[  n\right]
_{q}!}\\
&  =\frac{m}{t}%
{\displaystyle\sum\limits_{n=0}^{\infty}}
{\displaystyle\sum\limits_{k=0}^{n}}
\frac{1}{m^{n-k}}\left[
\begin{array}
[c]{c}%
n\\
k
\end{array}
\right]  _{q}\left(
{\displaystyle\sum\limits_{j=0}^{k}}
\left[
\begin{array}
[c]{c}%
k\\
j
\end{array}
\right]  _{q}\frac{(-1,q)_{k-j}}{m^{k-j}2^{k-j}}\mathfrak{G}_{j,q}\left(
x\right)  -\mathfrak{G}_{k,q}\left(  x\right)  \right)  \mathfrak{B}%
_{n-k,q}\left(  my\right)  \frac{t^{n}}{\left[  n\right]  _{q}!}\\
&  =%
{\displaystyle\sum\limits_{n=0}^{\infty}}
\dfrac{1}{\left[  n+1\right]  _{q}}%
{\displaystyle\sum\limits_{k=0}^{n+1}}
\frac{1}{m^{n-k}}\left[
\begin{array}
[c]{c}%
n+1\\
k
\end{array}
\right]  _{q}\left(
{\displaystyle\sum\limits_{j=0}^{k}}
\left[
\begin{array}
[c]{c}%
k\\
j
\end{array}
\right]  _{q}\frac{(-1,q)_{k-j}}{m^{k-j}2^{k-j}}\mathfrak{G}_{j,q}\left(
x\right)  -\mathfrak{G}_{k,q}\left(  x\right)  \right)  \mathfrak{B}%
_{n+1-k,q}\left(  my\right)  \frac{t^{n}}{\left[  n\right]  _{q}!}.
\end{align*}
The second identity can be proved in a like manner.
\end{proof}

\end{document}